\documentclass[11pt]{article}
\usepackage[T1]{fontenc}
\usepackage{lmodern}
\usepackage[margin=1in]{geometry}
\usepackage{amsmath,amssymb,amsthm,mathtools}
\usepackage{mathrsfs}
\usepackage{microtype}
\usepackage{xcolor}
\usepackage[colorlinks=true,linkcolor=blue!45!black,citecolor=blue!45!black,urlcolor=blue!45!black]{hyperref}
\hypersetup{pdftitle={Polynomial parametrization of SL2 over global function-field rings},pdfsubject={A proof using Larsen--Nguyen and Trost}}
\newtheorem{theorem}{Theorem}
\newtheorem{lemma}[theorem]{Lemma}

\theoremstyle{remark}

\newcommand{\SL}{\operatorname{SL}}
\newcommand{\Gal}{\operatorname{Gal}}
\newcommand{\tr}{\operatorname{tr}}
\newcommand{\Z}{\mathbb Z}
\newcommand{\F}{\mathbb F}
\newcommand{\cO}{\mathcal O}
\newcommand{\eps}{\varepsilon}

\def\F{\mathbb{F}}

\def\Z{\mathbb{Z}}

\def\cO{\mathcal{O}}

\def\cZ{\mathcal{Z}}

\def\fp{\mathfrak{p}}

\newcommand{\fl}{\mathfrak l}
\DeclareMathOperator{\Cl}{Cl}
\DeclareMathOperator{\Art}{Art}
\DeclareMathOperator{\Frob}{Frob}

\title{Polynomial parametrization of $\SL_2$\\over global function field rings}

\author{Dong Quan Ngoc Nguyen}
\date{October 8, 2026}

\begin{document}
\maketitle
\begin{abstract}
We combine the universal polynomial matrix identities of Larsen--Nguyen
with Trost's arithmetic lemmas to obtain polynomial parametrizations of
$\SL_2(\cO_{K,S})$ for every global function field $K$ and every nonempty
finite set of places $S$. The additional ingredient is a function-field
cancellation lemma, proved using ray class fields and Chebotarev in odd
characteristic and the perfectness of finite residue fields in characteristic
two. The proof also gives a polynomial matrix over $\Z$ that works
simultaneously for all these rings. 

\end{abstract}

\section{Statement and polynomial boundedness}

We follow the notion of polynomial boundedness from \cite{LN}. Let $R$ be a commutative ring. A subset $S\subset \SL_N(R)$ is \textit{polynomially bounded} if there exists an element $A(x_1, \ldots, x_n) \in \SL_N(\Z[x_1, \ldots, x_n])$ such that
\[
S \subset \{A(r_1, \ldots, r_n) | r_i \in R\},
\]
and $A(x_1, \ldots, x_n) = I$ has a solution in $R^n$, where $I$ denotes the identity matrix in $\SL_N(R)$. When $\SL_N(R)$ itself is polynomially bounded, we say it is \textit{polynomially parametrized}. 

A subset $\cZ \subseteq \SL_2(R) \times \SL_2(R)$ is said to be \textit{polynomially bounded} if there exist polynomially bounded sets $S, T \subseteq \SL_2(R)$ such that for all pairs $(M, N) \in \cZ$, there exist $X \in S$ and $Y \in T$ such that 
\[
N = XMY.
\]

An ordered pair $(a, b)$ is primitive if and only if the elements $a$ and $b$ generate the unit ideal. For primitive pairs $(a, b)$ and $(a', b')$, write $(a, b) \sim (a', b')$ if the set of pairs $(M, M') \in \SL_2(R) \times \SL_2(R)$ whose first rows are $(a \;\; b)$ and $(a' \; \; b')$, respectively is polynomially bounded.

For polynomials $P_1, Q_1, P_2, Q_2 \in \Z[t_1, \ldots, t_k]$ the relation $(P_1, Q_1) \sim (P_2, Q_2)$ indicates the following: (i) for any tuple $(a_1, \ldots, a_k) \in R^k$, the pair $(P_1(a_1, \ldots, a_k), Q_1(a_1, \ldots, a_k))$ is primitive if and only if the pair $(P_2(a_1, \ldots, a_k), Q_2(a_1, \ldots, a_k))$ is primitive; and (ii) the set of pairs $(X_1, X_2) \in \SL_2(R) \times \SL_2(R)$ such that for some $(a_1, \ldots, a_k) \in R^k$, the first row of $X_i$ is 
\[
(P_i(a_1, \ldots, a_k) \qquad Q_i(a_1, \ldots, a_k))
\]
for $i = 1, 2$ is polynomially bounded. Lemma 1.1 from \cite{LN} implies that the relation $\sim$ is an equivalence relation on $\Z[x_1, \ldots, x_k]^2$.

In Nguyen \cite{N}, the author proves that $\SL_2(\F_q[t])$ is polynomially parametrized, where $\F_q$ is the finite field of $q$ elements of odd characteristic. In this paper, we generalize this result to the ring of $S$-integers $\cO_{K, S}$ of a global function field $K$ in arbitrary positive characteristic and any nonempty finite set of places $S$ of $K$.

\begin{theorem}\label{thm:main}

There exists a polynomial matrix
\[
\mathcal P\in\SL_2\bigl(\Z[X_1,\ldots,X_N]\bigr)
\]
for some positive integer $N$ such that for every global function field $K$ and every nonempty finite
set of places $S$, the ring $R=\cO_{K,S}$ satisfies
\[
\mathcal P(R^{N})=\SL_2(R).
\]
In particular, $\SL_2(R)$ is polynomially parametrized.

\end{theorem}

We use the following ring-independent identity from
\cite[Proposition~1.3]{LN}:
\begin{equation}\label{eq:square}
(1+uv,v^2w)\sim(1+uv,w).
\end{equation}
It holds over every commutative ring, including rings of characteristic
two, and its published proof in \cite{LN} supplies fixed polynomial matrices over $\Z$.

Throughout the paper, $E_{ij}(a)$ denotes the elementary matrix whose diagonal entries are $1$, whose $(i, j)$-entry is $a$ for $i\ne j$, and whose other entries are $0$.

Rows with a unit entry are uniformly elementary reducible. For a unit $z$, put
\[
W(z):=E_{12}(z)E_{21}(-z^{-1})E_{12}(z).
\]
Then $W(z)W(-1)=\operatorname{diag}(z,z^{-1})$, where $z$ and $z^{-1}$ are viewed as independent parameters in the polynomial word, which has determinant one for arbitrary parameter values. Thus the diagonal matrices of the form $\operatorname{diag}(z,z^{-1})$ can be also written as a product of six elementary matrices.

\section{Selecting prime generators}
\label{sec-selecting-prime-generators}

In this section, we describe how to select prime generators in $R$ that will be useful in the proof of Lemma \ref{lem:cancel}. The analogue of Lemma \ref{lem:cancel} for the ring of integers in an imaginary quadratic field proved in Larsen and Nguyen \cite[Proposition 1.4]{LN} uses the Hasse theorem for primes in arithmetic progressions in number rings and the Hilbert reciprocity law. Our proof of Lemma \ref{lem:cancel} uses class field theory over global function fields instead.

Fix a nonzero proper ideal
\[
I=\prod_{\mathfrak{l}\mid I}\fl^{m_{\fl}}
\]
of $R$. Let $J_I$ be the group of nonzero fractional ideals of $R$ prime
to $I$. Its elements are products
\[
\mathfrak a=\prod_{\fp\nmid I}\fp^{n_{\fp}},\qquad n_{\fp}\in\Z,
\]
with only finitely many nonzero exponents.

Define the subgroup
\[
P_{I,1}=\{\gamma R:\gamma\in K^\times,\ \gamma\equiv1\pmod I\}.
\]
Since $\gamma$ need not belong to $R$, this congruence is interpreted
locally:
\[
v_{\fl}(\gamma-1)\geq m_{\fl}\qquad(\fl\mid I).
\]
Equivalently, $\gamma\in1+\fl^{m_{\fl}}R_{\fl}$ for each $\fl\mid I$.
In particular, $v_{\fl}(\gamma)=0$, so $\gamma R$ is indeed prime to $I$.

The \emph{$S$-ray class group modulo $I$} is
\begin{equation}\label{eq:ray-group}
\boxed{\Cl_I(R)=J_I/P_{I,1}.}
\end{equation}
Therefore
\begin{equation}\label{eq:equal-classes}
[\mathfrak a]=[\mathfrak b]
\quad\Longleftrightarrow\quad
\mathfrak a\mathfrak b^{-1}=\gamma R
\text{ for some }\gamma\equiv1\pmod I.
\end{equation}

\subsection{Artin isomorphism}

Let $\F_q$ denote the full constant field of $K$. Class field theory associates to \eqref{eq:ray-group} a finite abelian
extension $H_I/K$ and an isomorphism called the \textit{Artin isomorphism}
\begin{equation}\label{eq:artin}
\boxed{\Art_I:\Cl_I(R)\xrightarrow{\sim}\Gal(H_I/K).}
\end{equation}
The field $H_I$ is the $S$-ray class field modulo $I$. More precisely, put
$\mathfrak m=\sum_{\fl\mid I}m_{\fl}\fl$. Within a fixed algebraic closure,
$H_I$ is the maximal abelian extension whose conductor is bounded by
$\mathfrak m$ and in which all places of $S$ split completely.
In particular, it is unramified at all primes of $R$ not dividing $I$.
The conductor bound controls the depth of ramification at primes dividing
$I$; the condition ``unramified outside $I$'' alone does not specify this.
See Auer~\cite[Section~1]{Auer}, whose notation $K_S^{\mathfrak m}$
corresponds to $H_I$.

Let $\fp\nmid I$ be a prime of $R$, and choose a place $\mathfrak P$ of
$H_I$ above $\fp$. Write
\[
N\fp=|R/\fp|=q^{\deg\fp}.
\]
Since $\fp$ is unramified, the arithmetic Frobenius at $\mathfrak P$ is
the unique element of its decomposition group inducing
\[
\overline z\longmapsto\overline z^{\,N\fp}
\]
on the residue field at $\mathfrak P$. Equivalently,
\[
\Frob_{\mathfrak P/\fp}(z)\equiv z^{N\fp}\pmod{\mathfrak P}
\]
for elements $z$ integral at $\mathfrak P$. Because the extension is
abelian, Frobenius is independent of the chosen place above $\fp$; denote
it by $\Frob_{\fp}(H_I/K)$.

The Artin map is characterized by
\[
\Art_I([\fp])=\Frob_{\fp}(H_I/K).
\]
It extends multiplicatively:
\[
\Art_I\left(\left[\prod_{\fp}\fp^{n_{\fp}}\right]\right)
=\prod_{\fp}\Frob_{\fp}(H_I/K)^{n_{\fp}}.
\]
The reciprocity theorem says that the corresponding map on $J_I$ is
surjective with kernel exactly $P_{I,1}$. Thus it induces the Artin isomorphism
\eqref{eq:artin}. See Rosen~\cite[Chapter~9]{Rosen} for Artin reciprocity.

\subsection{From a prescribed ray class to a prime generator}

Let $0\ne d\in R$ satisfy $dR+I=R$. We seek $d'\in R$ such that
\[
d'R\text{ is prime},\qquad d'\equiv d\pmod I.
\]
Set $\sigma_d=\Art_I([dR])$. Chebotarev's theorem for global function fields (see \cite[Theorem 9.13A]{Rosen}) supplies infinitely many primes
$\fp\nmid I$ with
\[
\Frob_{\fp}(H_I/K)=\sigma_d
\]
(see Rosen~\cite[Theorem~9.13A]{Rosen}). Because \eqref{eq:artin} is
injective, this is equivalent to
\[
[\fp]=[dR]\quad\text{in }\Cl_I(R).
\]
By \eqref{eq:equal-classes}, there is $\gamma\in K^\times$ such that
\[
\fp(dR)^{-1}=\gamma R,\qquad \gamma\equiv1\pmod I.
\]
Hence $\fp=d\gamma R$. Define $d'=d\gamma$.

Two details justify the conclusion. First, $d'R=\fp\subseteq R$, so
$d'\in R$. Second, for every $\fl\mid I$, the element $d$ is a local
unit, whence
\[
v_{\fl}(d'-d)=v_{\fl}\bigl(d(\gamma-1)\bigr)
\geq m_{\fl}.
\]
Since $d'-d\in R$, these inequalities imply $d'-d\in I$. Thus
\begin{equation}\label{eq:generator}
\boxed{d'R=\fp,\qquad d'\equiv d\pmod I.}
\end{equation}

Chebotarev's theorem initially gives prime \emph{ideals}. The choice of the ray
class $[dR]$ forces them to be principal and supplies a generator in the
desired residue class. Since infinitely many such primes exist, any
fixed finite set of primes may be excluded.

\subsection{Adding a splitting condition in another extension}
Suppose $L/K$ is finite abelian and $L\cap H_I=K$. The restriction maps give
\begin{equation}\label{eq:product}
\Gal(H_IL/K)\xrightarrow{\sim}
\Gal(H_I/K)\times\Gal(L/K).
\end{equation}
Choose the automorphism corresponding to $(\sigma_d,1)$. Applying
Chebotarev's theorem to $H_IL/K$, and excluding ramified primes, gives infinitely
many primes $\fp$ satisfying
\[
\Art_I([\fp])= \Frob_{\fp}(H_I/K)=\sigma_d,
\qquad \Frob_{\fp}(L/K)=1.
\]
The first condition gives \eqref{eq:generator}; the second says that
$\fp$ splits completely in $L$. We therefore obtain
\[
\boxed{d'\equiv d\pmod I,\qquad
d'R\text{ is prime and splits completely in }L.}
\]

The intersection hypothesis guarantees compatibility. In general, with
$E=L\cap H_I$, the image of the restriction map consists of pairs
$(\sigma,\tau)$ satisfying $\sigma|_E=\tau|_E$. Thus a prescribed
automorphism of $H_I$ and the identity on $L$ can be combined only if the
former acts trivially on $E$. When $E=K$, there is no extra condition.

\section{The function-field cancellation lemma}

\begin{lemma}
\label{lem:cancel}
For $R=\cO_{K,S}$ as above, one has
\begin{equation}\label{eq:cancel}
(1+ab,bc)\sim(1+ab,c).
\end{equation}
The polynomial families realizing this equivalence can be chosen
independently of $K$ and $S$.
\end{lemma}

\begin{proof}
Both rows are primitive simultaneously. Put $x=1+ab$. If $x$ is zero
or a unit, each primitive row in question has a unit entry, so the
elementary observation above applies. Thus assume that $x$ is nonzero and
nonunit. This, in particular, implies that $a$ is nonzero.

Two changes may be made simultaneously on both sides of
\eqref{eq:cancel}. First, replace $c$ by $c+xd$: the second coordinates
change by $xbd$ and $xd$, respectively. Second, replace $a$ by $a+ce$:
the first coordinate changes by $bce$, an elementary change using the
second coordinate on either side. These use a fixed number of elementary
operations.

We contend that there exists an element $r \in R$ such that $a \equiv -r^2 \pmod{cR}$. 

Suppose first that $\operatorname{char}K=2$. Replacing $c$ by a prime
generator in its residue class modulo $xR$, we see that the residue field $R/cR$
is finite and hence perfect. Consider the Frobenius map $F : R/cR \to R/cR$ that sends each element $u$ to $u^2$. Since $R$ is of characteristic $2$, we see that $F(u) = u^2 = v^2 = F(v)$ implies that $(u - v)^2 = 0$, and hence $u = v$ in the field $R/cR$. Thus $F$ is injective. Since $R/cR$ is finite, $F$ is also surjective, and thus $F$ is an isomorphism. Therefore there exists an element $r \in R$ such that $-a \equiv F(\bar r) \equiv r^2 \pmod{cR}$, and thus $a \equiv - r^2 \pmod{cR}$.

 Thus there is $r\in R$ with
$a\equiv-r^2\pmod{cR}$.

Now suppose that $\operatorname{char}K\ne2$. Since $a\ne 0$, there are only finitely many prime divisors of $a$. We first replace $c$ by a
prime generator in its residue class modulo $xR$ that is relatively prime to the prime
divisors of $a$. Then $aR + cR = R$, and thus one can replace $a$ by a prime generator
in its residue class modulo $cR$. Denote the resulting elements by the same symbols $a,c$, respectively, and put $x=1+ab$ again. If the new $x$ is a unit, then we can repeat the same arguments at the beginning of the proof. Otherwise, we observe that the quadratic extension
\[
L=K(\sqrt{-a})
\]
is ramified at the prime $aR$, since $aR$ is prime and the valuation of $-a$ at $aR$
is one (see \cite[Proposition 10.3]{Rosen}). On the other hand, $aR$ is prime to $xR$, and thus the ray class field 
$H_{xR}/K$ is unramified at $aR$. Therefore $L\cap H_{xR}=K$.
The prime-selection observation in Section \ref{sec-selecting-prime-generators} gives a prime generator $\pi$ in $R$ such that
\[
\pi\equiv c\pmod{xR},\qquad \pi R\text{ splits completely in }L,
\]
with $\pi R\ne aR$. Replace $c$ by $\pi$, and rename $\pi$ to be $c$ again. The prime $cR$ splits in $L$ if and only if $-a$ is a
square modulo $cR$. Thus $a\equiv-r^2\pmod{cR}$ for some $r\in R$.

In either characteristic, write $a+r^2=ch$ for some element $h \in R$. Then
\begin{align*}
(1+ab,bc)
&\sim(1-r^2b,bc)\\
&\sim\bigl(1-r^2b,bc-(1-r^2b)bc\bigr)\\
&=\bigl(1-r^2b,(rb)^2c\bigr)\\
&\sim(1-r^2b,c)\\
&\sim(1+ab,c).
\end{align*}
The penultimate equivalence uses \eqref{eq:square} with $u=-r$ and
$v=rb$, and the last equivalence follows from $a + r^2 = ch$ and the following elementary operation
\[
(1 - r^2b, c)E_{21}(bh) = (1 - r^2b + bch, c) = (1 + ab, c).
\]

 Every operation uses a fixed polynomial family. The chosen
primes and square roots affect parameter values only.
\end{proof}

This replaces the number-field argument of
\cite[Proposition~1.4]{LN}. Ramification at $aR$ ensures compatibility
of the splitting and congruence conditions, with no extra conditions at $S$.

\section{Power comparison with a uniform bound}

The following lemma is an analogue of \cite[Proposition 1.5]{LN} in the function field setting.

\begin{lemma}
\label{lem:power}

Let $M=\begin{pmatrix}a&b\\c&d\end{pmatrix}\in\SL_2(R)$, and let
$(a_n,b_n)$ be the first row of $M^n$. Then
\begin{equation}
\label{eq:power}
(a_n,b_n)\sim(a^n,b),\qquad n\ge0.
\end{equation}

\end{lemma}

\begin{proof}

Define $Q_{-1}(T) = -1$, $Q_0(T)=0$, $Q_1(T)=1$, and
$Q_{j+1}(T)=TQ_j(T)-Q_{j-1}(T)$. Set $u_j=Q_j(\tr M)$ for each $j \ge - 1$. By induction on $j$, we see that
\begin{align}
\label{eq:eq1-in-power-lem}
Q_{j + 1}(T)Q_{j - 1}(T) = Q_j(T)^2 - 1.
\end{align}
for $j \ge 0$.

By (\ref{eq:eq1-in-power-lem}), there are polynomials $V_n,W_n\in\Z[T]$ such that
\begin{equation}\label{eq:factor}
Q_n=V_nW_n,\qquad V_n\mid Q_{n-1}-1,\qquad
W_n\mid Q_{n-1}+1.
\end{equation}
Explicit choices are
\[
\begin{array}{c|cc}
n&V_n&W_n\\ \hline
2k+1&Q_{k+1}-Q_k&Q_{k+1}+Q_k\\
2k\ (k\ge1)&Q_{k+1}-Q_{k-1}&Q_k.
\end{array}
\]
These identities follow by induction from the recurrence. In particular,
they do not require unique factorization in $R$ and remain valid in
characteristic two.

By the Cayley--Hamilton theorem, and by induction on $n$, we have 
\begin{equation}
\label{eq:eq2-power}
M^n=u_nM-u_{n-1}I\qquad(n\ge1).
\end{equation}

Set $v=V_n(\tr M)$ and $w=W_n(\tr M)$. Then it follows from (\ref{eq:factor}) and (\ref{eq:eq2-power}) that
\[
b_n=bu_n = bvw,\qquad a_n\equiv-1\pmod{vR},\qquad
a_n\equiv1\pmod{wR}.
\]
 Since for arbitrary elements $\alpha, \beta \in R$, 
\begin{align*}
(\alpha, \beta)(E_{12}(1)E_{21}(-1)E_{12}(1))^2 = (-\alpha, -\beta),
\end{align*}
and thus
\[
(\alpha, \beta) \sim (-\alpha, -\beta).
\]

Hence applying Lemma~\ref{lem:cancel} to cancel $w$, then changing signs and
applying it again to cancel $v$, we obtain 
\[
(a_n,b_n)\sim(a_n,bv)\sim(-a_n,-bv)
\sim(-a_n,-b)\sim(a_n,b).
\]
Modulo $bR$, the matrix $M$ is lower triangular, so $a_n\equiv a^n$
modulo $bR$. One more elementary change proves \eqref{eq:power}.
The case $n=0$ is immediate.

\end{proof}

\section{Using Trost's lemmas}

Let $\F_q$ be the full constant field of $K$, and put $m=q-1$.
For a nonzero nonunit $z\in R$, we denote by $\eps(z)$ the exponent of the multiplicative group of units of $R/zR$, i.e., the smallest positive integer $e$ such that 
\[
u^e = 1
\]
for every element $u \in (R/zR)^{\times}$.

We use the following statements from Trost \cite{Trost}:
\begin{enumerate}
\item Lemma~3.1: up to at most three multiplications by elementary matrices, the
first row of a matrix in $\SL_2(R)$ is $(a^m,b)$ with $bR$ prime.

\item Lemma~3.3: if $bR$ is prime and $aR + bR =R$, one can choose $c$
with $bc\equiv-1\pmod{aR}$ and
$\gcd(\eps(b),\eps(c))=m$.
\end{enumerate}
Unit choices of $c$, if they occur, will be handled directly below.

These results apply to $\cO_{K,S}$. Indeed, by the Riemann--Roch theorem (see \cite{Rosen}), one may
choose a nonconstant function $t$ whose pole set is exactly $S$.
Then $R$ is the integral closure of $\F_q[t]$ in $K$, which is the
setup of Trost's paper; see also \cite[Remark~2.5]{Trost}.

\begin{proof}[Proof of Theorem~\ref{thm:main}]

Let $A\in\SL_2(R)$. By Trost's first lemma, a modification of $A$, denoted by
$A'$, using elementary matrices  has first row $(a^m,b)$ with $bR$ prime. If $a$ is a unit,
the row can be transformed into $(1, 0)$ by a product of a uniformly bounded number of elementary matrices. The case $a=0$
would force $b$ to be a unit since $aR + bR = R$, and so cannot occur with
$bR$ being prime. Thus assume $a$ is a nonzero nonunit.

Choose $c$ by Trost's second lemma and define
\[
d=\frac{1+bc}{a}\in R,\qquad
B=\begin{pmatrix}a&b\\c&d\end{pmatrix}\in\SL_2(R).
\]
If $c$ is a unit, apply Lemma~\ref{lem:power} to the transpose $B^T$ with exponent
$m$. Its comparison row $(a^m,c)$ has a unit entry, so $(B^T)^m$
and hence $B^m$ are polynomially bounded. The final comparison below
then completes this case. Also $c=0$ is impossible since
$bc\equiv-1\pmod{aR}$ and $aR$ is proper.

Assume therefore that $c$ is a nonzero nonunit. Since
$\gcd(\eps(b),\eps(c))=m$, there are positive integers $s,t$ such that $\eps(b)$ divides $t$, $\eps(c)$ divides $s$, and
\begin{equation}
\label{eq:exponents}
t-s=m.
\end{equation}

The determinant equation ensures that $a$ is invertible modulo both
$bR$ and $cR$. Consequently
\[
a^t\equiv1\pmod{bR},\qquad a^s\equiv1\pmod{cR}.
\]
A row of the form $(1+zh,z)$ reduces to $(1,0)$ in two elementary
operations. Applying Lemma~\ref{lem:power} to $B$ and $B^T$ shows
that $B^t$ and $(B^T)^s$ belong to fixed polynomially bounded families.
Taking transposes and inverses, the same is true of $B^{-s}$.
Therefore
\[
B^m=B^tB^{-s}
\]
belongs to a fixed polynomially bounded set in $\SL_2(R)$.

Finally, Lemma~\ref{lem:power} with exponent $m$ identifies the first
row of $B^m$, up to bounded equivalence, with $(a^m,b)$, the first
row of $A'$. This implies that $A'$ belongs to a fixed polynomially bounded set. Since $A'$ is obtained from $A$, using up to three elementary operations by Trost's first lemma, $A$ also belongs to a fixed polynomially bounded set in $\SL_2(R)$. Since $A$ is arbitrary in $\SL_2(R)$, $\SL_2(R)$ is polynomially parametrized.

All polynomial matrices used have integer coefficients and are drawn
from a fixed finite collection of universal families, their inverses
and transposes, and fixed products. The arithmetic choices, including
$m,s,t$, affect witness values only. The finitely many branches can
be included in one product of families containing the identity.
Thus a single polynomial matrix works for every $K,S$, as asserted.

\end{proof}


\begin{thebibliography}{9}


\bibitem{Auer}
R.~Auer,
\emph{Ray class fields of global function fields with many rational places},
Acta Arithmetica \textbf{95} (2000), 97--122.
\href{https://arxiv.org/abs/math/9803065}{arXiv:math/9803065}.

\bibitem{Goss}
D.~Goss,
\emph{Basic Structures of Function Field Arithmetic}, Ergebnisse der Mathematik und ihrer Grenzgebiete, Band 35, Springer, 1998.


\bibitem{LN}
M.~Larsen and D.~Q.~N.~Nguyen,
\emph{Polynomial Parametrization for $\SL_2$ over Quadratic Number Rings},
International Mathematics Research Notices \textbf{2021} (2021), no.~9,
6993--7003.
\href{https://doi.org/10.1093/imrn/rnz046}{doi:10.1093/imrn/rnz046}.



\bibitem{N}
D.~Q.~N.~Nguyen,
\emph{Certain sets over function fields are polynomial families}, Annali Scuola Normale Superiore--Classe Di Scienze, 17(1), 205--230. 


\bibitem{Rosen}
M.~Rosen,
\emph{Number Theory in Function Fields}, Graduate Texts in Mathematics 210, Springer, 2002.




\bibitem{Trost}
A.~A.~Trost,
\emph{Elementary bounded generation for $\SL_n$ for global function
fields and $n\ge3$}, Bulletin of the London Mathematical Society, Volume 56, Issue 1, 219--231.


\end{thebibliography}
\end{document}